\documentclass[11pt]{amsart}
\usepackage{amssymb}
\usepackage{graphicx} 
\usepackage{dbnsymb}
\usepackage{enumerate}
\usepackage{mathtools, multicol}
\usepackage{color,soul}
\usepackage{cite}
\usepackage{tikz}
\usepackage{hyperref}
\usetikzlibrary{matrix}
\usepackage{caption}
\usepackage{amsmath, amsthm}
\usepackage{mathtools}
 \usepackage{relsize}
 \usetikzlibrary{cd}
 \usetikzlibrary{decorations.pathreplacing}
 \usepackage{tikz,calc}
\usepackage{color}
\usepackage[margin=1.1in]{geometry}

\usepackage{multicol}

\usepackage{wrapfig}
\usepackage{cutwin}

\usepackage{lmodern}
\usepackage{enumitem}
\usepackage{stackengine}
\usepackage{appendix}
\usepackage{scrextend}

\newtheorem{thm}{Theorem}[section]

\newtheorem{prop}[thm]{Proposition}

\newtheorem{GPVthm}[thm]{GPV Theorem}

\newcommand{\theoremname}{Theorem:}

\newtheorem*{conj*}{Conjecture}

  \theoremstyle{definition}
  \newtheorem{defn}[thm]{Definition}

  \newtheorem*{claim*}{Claim}

  \newtheorem*{question*}{Question}
  \newtheorem*{answer*}{Answer}
  \newtheorem*{application*}{Application}

  \theoremstyle{remark}
  \newtheorem{rmk}[thm]{Remark}
  \newtheorem*{rmk*}{Remark}

\usepackage{bbm}
\newcommand{\Q}{\mathbb{Q}}
\newcommand{\R}{\mathbb{R}}

\newcommand{\K}{\mathbb{K}}

\newcommand{\GD}{\mathcal{GD}}
\newcommand{\LGD}{\mathcal{LGD}}

\newcommand{\IntersectingArcs}{
 \begin{tikzpicture}[scale=.2, baseline=.0mm  ]
\draw[](-2.9,0)--(2.5,0);
\draw[->] (2,0) arc (0:180:1.5);
\draw[->] (.5,0) arc (0:180:1.5);
\end{tikzpicture}
}

\newcommand{\OneArrow}{
 \begin{tikzpicture}[scale=.2, baseline=.0mm  ]
\draw[](-2.5,0)--(2.5,0);
\draw[->] (1.75,0) arc (0:180:1.75);
\end{tikzpicture}
}

\newcommand{\OneOtherArrow}{
 \begin{tikzpicture}[scale=.2, baseline=.0mm  ]
\draw[](-2.5,0)--(2.5,0);
\draw[<-] (1.75,0) arc (0:180:1.75);
\end{tikzpicture}
}
\newcommand{\OneArrowdashed}{
 \begin{tikzpicture}[scale=.2, baseline=.0mm  ]
\draw[](-2.5,0)--(2.5,0);
\draw[->,dashed] (1.75,0) arc (0:180:1.75);
\end{tikzpicture}
}

\newcommand{\OneOtherArrowdashed}{
 \begin{tikzpicture}[scale=.2, baseline=.0mm  ]
\draw[](-2.5,0)--(2.5,0);
\draw[<-,dashed] (1.75,0) arc (0:180:1.75);
\end{tikzpicture}
}

\newcommand{\LeftArrow}[1]{
 \begin{tikzpicture}[scale=.2, baseline=.0mm  ]
  \node at (-3,0) {#1};
\draw[](-2.5,0)--(2.5,0);
\draw[->] (1.75,0) arc (0:180:1.75);
\end{tikzpicture}
}

\newcommand{\RightArrow}[1]{
 \begin{tikzpicture}[scale=.2, baseline=.0mm  ]
 \node at (-3,0) {#1};
\draw[](-2.5,0)--(2.5,0);
\draw[<-] (1.75,0) arc (0:180:1.75);
\end{tikzpicture}
}
\newcommand{\LeftArrowdashed}[1]{
 \begin{tikzpicture}[scale=.2, baseline=.0mm  ]
  \node at (-3,0) {#1};
\draw[](-2.5,0)--(2.5,0);
\draw[->,dashed] (1.75,0) arc (0:180:1.75);
\end{tikzpicture}
}

\newcommand{\RightArrowdashed}[1]{
 \begin{tikzpicture}[scale=.2, baseline=.0mm  ]
  \node at (-3,0) {#1};
\draw[](-2.5,0)--(2.5,0);
\draw[<-,dashed] (1.75,0) arc (0:180:1.75);
\end{tikzpicture}
}

\newcommand{\OneArc}{
 \begin{tikzpicture}[scale=.2, baseline=.0mm  ]
\draw[](-2.5,0)--(2.5,0);
\draw[-] (1.75,0) arc (0:180:1.75);
\end{tikzpicture}
}

\newcommand{\linkArcsolidup}[2]{
 \begin{tikzpicture}[scale=.2, baseline=.0mm  ]
 \node at (-2.3,-1) {#1};
 \node at (-2.3,1.5) {#2};
\draw[](-1.5,-1)--(1.5,-1);
\draw[](-1.5,1.5)--(1.5,1.5);
\draw[->] (0,-1)--(0,1.5);
\end{tikzpicture}
}
\newcommand{\linkArcsoliddown}[2]{
 \begin{tikzpicture}[scale=.2, baseline=.0mm  ]
  \node at (-2.3,-1) {#1};
 \node at (-2.3,1.5) {#2};
\draw[](-1.5,-1)--(1.5,-1);
\draw[](-1.5,1.5)--(1.5,1.5);
\draw[<-] (0,-1)--(0,1.5);
\end{tikzpicture}
}
\newcommand{\linkArcdashup}[2]{
 \begin{tikzpicture}[scale=.2, baseline=.0mm  ]
  \node at (-2.3,-1) {#1};
 \node at (-2.3,1.5) {#2};
\draw[](-1.5,-1)--(1.5,-1);
\draw[](-1.5,1.5)--(1.5,1.5);
\draw[->, dashed] (0,-1)--(0,1.5);
\end{tikzpicture}
}
\newcommand{\linkArcdashdown}[2]{
 \begin{tikzpicture}[scale=.2, baseline=.0mm  ]
  \node at (-2.3,-1) {#1};
 \node at (-2.3,1.5) {#2};
\draw[](-1.5,-1)--(1.5,-1);
\draw[](-1.5,1.5)--(1.5,1.5);
\draw[<-,dashed] (0,-1)--(0,1.5);
\end{tikzpicture}
}

\newcommand{\pcross}{
\begin{tikzpicture}[baseline=-2.75, scale=.2]
\draw[-{Stealth[ length=1.25mm, width=1.25mm]},thick ](-1,-1)--(1,1);
\draw[-{Stealth[ length=1.25mm, width=1.25mm]},thick ](-.25,.25)--(-1,1);
\draw[thick ](.3,-.3)--(1,-1);
\end{tikzpicture}}

\newcommand{\ncross}{
\begin{tikzpicture}[baseline=-2.75,scale=.2]
\draw[-{Stealth[ length=1.25mm, width=1.25mm]},thick ](1,-1)--(-1,1);
\draw[-{Stealth[ length=1.25mm, width=1.25mm]},thick ](.25,.25)--(1,1);
\draw[thick ](-.3,-.3)--(-1,-1);
\end{tikzpicture}}

\newcommand{\smooth}{
\begin{tikzpicture}[baseline=-2.75,scale=.2]
\draw[-{Stealth[ length=1.25mm, width=1.25mm]},thick ] (-.8,-1) to[out=45,in=-45] (-1,1);
\draw[-{Stealth[ length=1.25mm, width=1.25mm]},thick ] (.8,-1) to[out=135,in=215] (1,1);
\end{tikzpicture}}

\title[Short Header Title]{Your Long, Detailed Document Title}

\title[ GPV formulas for Low Degree Coefficients of the Conway Polynomial]{Combinatorial Goussarov-Polyak-Viro Formulas for the Linking Number and Low Degree Coefficients of the Conway Polynomial}
\author{Nancy Scherich and Nathaniel Song}
\date{}

\begin{document}

\begin{abstract}
    The linking number and coefficients of the Conway polynomial  are two examples of finite type invariants. Goussarov-Polyak-Viro proved that any finite type knot invariant can be computed in a two step process, where the second step is referred to as the \emph{GPV map} for the invariant. This paper computes the GPV maps for the linking number and low-degree coefficients of the Conway Polynomial. Chmutov--Khoury--Rossi gave one description of the GPV maps for coefficients of the Conway polynomial in terms of arrow diagrams and state-sum calculations \cite{CKR}. We take a much more grounded approach using fundamental linear algebra  to describe the GPV maps in terms of a basis for the vector space of Gauss diagrams.
\end{abstract}

\maketitle

\section{Introduction}

A \emph{knot} is a closed curve in $\R^3$ with no self-intersections. A \emph{link} is a collection of two or more knots that may intertwine with each other. 
Studying knots from a topological perspective, we think of knots as being made of infinitely thin and infinitely stretchable material.
Two knots (or links) are equivalent, or ``the same", if one can be continuously deformed into the other  without cutting  or tearing. Formally, this is called \emph{ambient isotopy}. 
A knot (or link) \emph{diagram} is a projection of the knot onto a plane, with crossing information recorded to distinguish which strand passes over or under the other. See Figure \ref{fig:example knot diagrams} for example knot diagrams.
Projecting a knot in different directions, or performing an ambient isotopy before projecting, can lead to wildly different diagrams all representing the same knot. This often makes it difficult to distinguish knots and knot diagrams from each other.

\begin{figure}[h]
    \centering
    \includegraphics[width=0.35\linewidth]{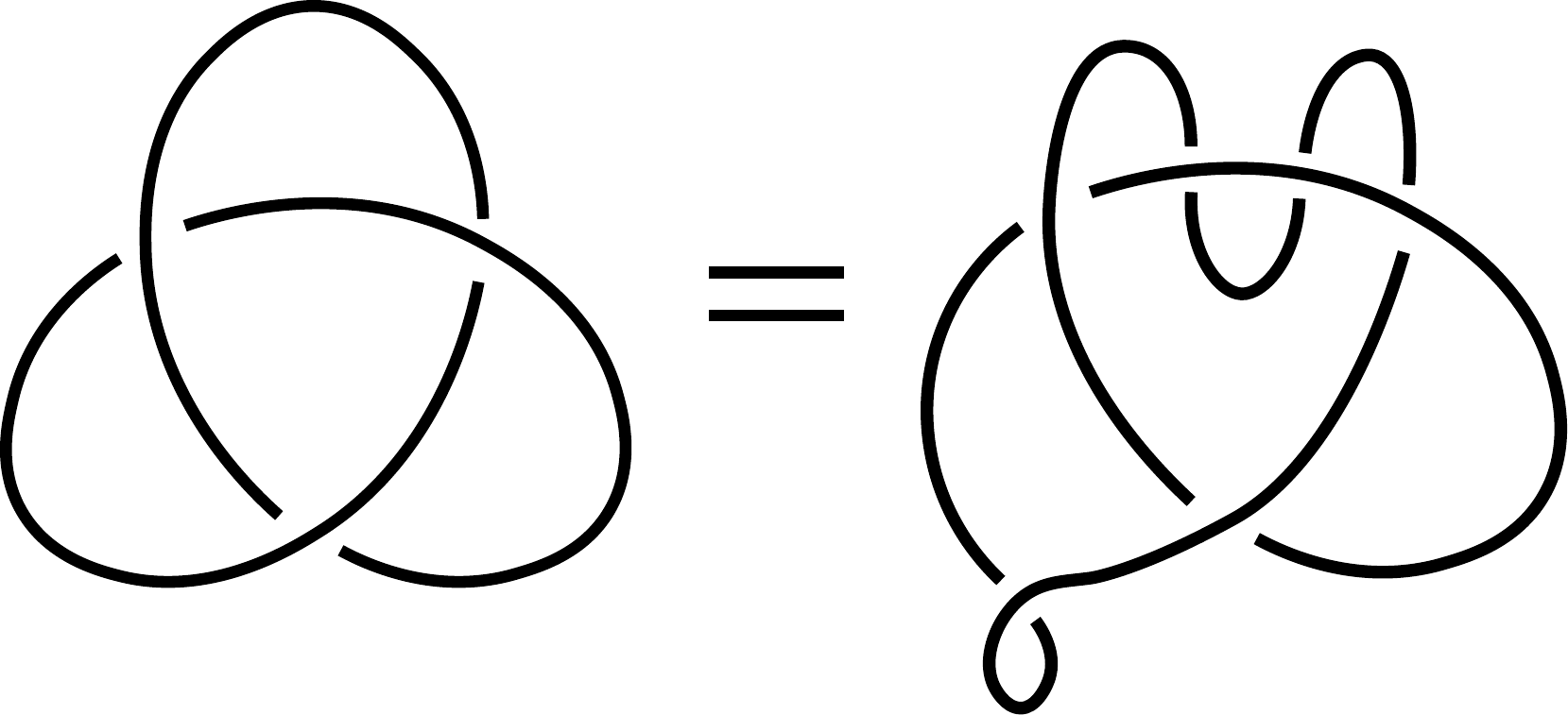}
    \caption{An example of two equivalent knot diagrams.} 
    \label{fig:example knot diagrams}
\end{figure}

The primary tools used to distinguish and classify knots are called \emph{knot invariants}. Invariants assign a value to a knot, and if two knots are ambient isotopic to each other (``the same knot"), the invariant must output the same value.
 If two knots have different invariant values, we can determine that they are truly different knots; however, when two knots share the same value from an invariant, we cannot conclude that they are the same knot. Invariants can output a number, for example the linking number or minimal crossing number, or can output a boolean (true or false) like the tri-colorability invariant. The most common type is a polynomial invariant which outputs polynomial values, for example the Jones or  Conway polynomial invariants.

There is a special class of knot invariants called \emph{finite type invariants}, also known as Vassiliev invariants, which satisfy a vanishing condition when extended to singular knots via the Vassiliev skein relation \cite{BN1}. 
\emph{Singular knot diagrams} are knot diagrams that have three types of types of crossings:  $\undercrossing$, $\overcrossing$, and a third option $\doublepoint$ called a \emph{double point}.
A knot invariant $f$ is said to be of \emph{finite type $k$} if it vanishes on all knots with at least $k+1$ double points, where $f$ is extended to singular knots by the Vassiliev skein relation:
\begin{equation}\label{vassiliev}f(\doublepoint)=f(\overcrossing) - f(\undercrossing).\end{equation}
This formula should be interpreted locally. That is, start with a knot $K$ and identify one crossing, or double point, in $K$ and call it $x$. Keeping all of $K$ fixed outside $x$, change whatever crossing was originally at $x$ to be a positive crossing, and call this new diagram $K_+$. Similarly, let $K_-$ and $K_\bullet$ be $K$ with $x$ replaced by a negative crossing or double point, respectively. Then Equation \ref{vassiliev} equates to 
\[f(K_\bullet)=f(K_+) - f(K_-).\]

Finite type invariants underlie many of the classical knot invariants; for instance, they include the coefficients of the Jones, Conway, and more generally HOMFLY-PT polynomials \cite{BL93, BN1}.
For example, the linking number of a two-component knot is a finite type invariant of type 1. 

Goussarov-Polyak-Viro (GPV) proved that all finite type invariants can be computed in a combinatorial way using Gauss diagrams \cite{GPV}. In  Section \ref{sec:GDandGPV}, we describe Gauss diagrams and the GPV result in more detail. To summarize, Goussarov-Polyak-Viro proved that any finite type knot invariant $f$ of type $k$ can be computed in a two-step process where the first step is the same regardless of $f$ (it only depends on $k$), and the second step is a map which we denote $\omega_f$ and call the GPV map for $f$. GPV maps are linear transformations on the vector space of Gauss diagrams.

The purpose of this paper is to compute the GPV maps for the linking number and low-degree coefficients of the Conway Polynomial. Chmutov--Khoury--Rossi gave one description of the GPV maps for coefficients of the Conway polynomial in terms of arrow diagrams and state-sum calculations \cite{CKR}. We take a much more grounded approach using fundamental linear algebra  to describe of the GPV maps in terms of a basis. Our main results are stated in Theorem \ref{thm:main_linking} and Theorem \ref{thm:main_Conway}.

\subsection{Acknowledgments.} This project is part of an undergraduate research experience at Elon University led by the first author with students Joshua Miller and the second author. We thank Joshua Miller for helpful contributions to the linking number section of this paper. This project is sponsored in part by NSF DMS-2532699 supporting the first author.

\section{Gauss diagrams and the GPV Theorem}\label{sec:GDandGPV}

The crossing information of a knot can be distilled and stored in combinatorial objects called \emph{Gauss Diagrams}.
Gauss diagrams depend on a parameterization of a knot and therefore require information on where the parameterization starts and stops. There are two equivalent methods to present this information: an oriented tangle with two ends escaping to $\pm \infty$, or the knot diagram includes an orientation and marked starting point via a dot on an arc.  Either scenario will be referred to as an \emph{oriented long knot diagram}, see Figure \ref{fig:GD Diagram.png} for an example.

\begin{figure}
    \centering
    \includegraphics[width=0.9\linewidth]{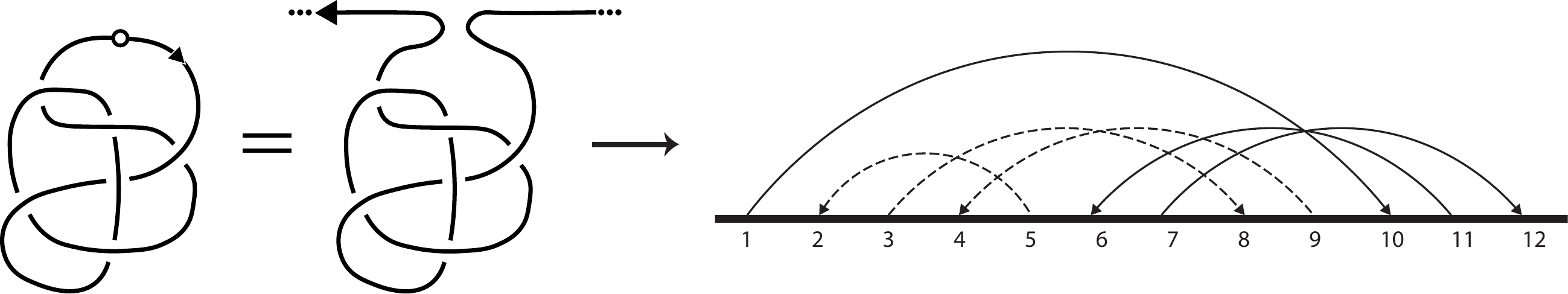}
    \caption{Example of a Gauss Diagram of a long knot.}
    \label{fig:GD Diagram.png}
\end{figure}

\begin{defn}[Gauss Diagram]
Let $K$ be an oriented long knot diagram parameterized by an interval $I\subseteq\mathbb{R}$ with $n$ crossings occurring at integer locations of the parameterization.
A \emph{Gauss diagram} of $K$ is given by the interval $I$ along with $n$ arrows, one for each crossing.  
The head of an arrow is at the point  which parametrizes the lower strand of the crossing, and the tail of the arrow is at the point which parametrizes the upper strand of the crossing. 
Each arrow is styled corresponding to the sign of the crossing: dashed for positive, solid for negative. (One can also define a Gauss diagram for a link, see Section \ref{sec:linking}.)
\end{defn} 
 We follow the convention that the interval parameterizing line of a Gauss diagram is always oriented from left to right unless otherwise marked. 
\begin{rmk}
    Changing the starting point for the parameterization changes the Gauss diagram by cyclically permuting the arrows along the parameterization line. So there is not a unique choice of Gauss diagram for a knot. However, each oriented long knot has only one Gauss diagram.
\end{rmk}

An example Gauss diagram for a long knot is shown in Figure \ref{fig:GD Diagram.png}.
Let $\GD_k$ denote the vector space of all $\Q$-linear combinations of Gauss diagrams with at most $k$ arrows. Here, scalar multiplication and diagram addition are formal and do not have inherent knot-theoretic meaning. An example element of $\GD_4$ is shown in Figure \ref{fig:k1_basis}.
When $k=0$, $\GD_0$ is a 1-dimensional space generated by the empty diagram, i.e., $\GD_0\cong \Q$.
When $k=1$, $\GD_1$ is 4-dimensional with basis 
\OneArrow, \OneOtherArrow, \OneArrowdashed, and \OneOtherArrowdashed.

In general, the number of distinct $k$-arrow diagrams can be computed by the formula $$2^{2k}\cdot \frac{(2k)!}{(2!)^kk!}.$$ Here, $\frac{(2k)!}{(2!)^kk!}$ is the total number of possible ways to arrange $k$ arrows (without heads) on a line with $2k$ endpoints (i.e., the total number of perfect matchings of a complete graph with $2k$ vertices, see \cite{stack}). Multiplying by $2^{2k}$ accounts for arrow-head orientation and positive/negative signs.
Since $\GD_k$ has all diagrams with $k$ \emph{or less} arrows, the dimension of $\GD_k$ is given by the following formula
 $$dim_\Q(\GD_k)=\sum_{m=1}^k 2^{2m}\frac{(2m)!}{(2!)^m m!}.$$

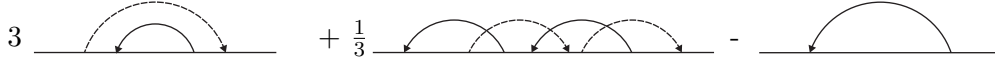
\begin{figure}[]
    \centering
    \begin{picture}(190, 50)
        \put(-100, 0){\includegraphics[width=0.8\linewidth]{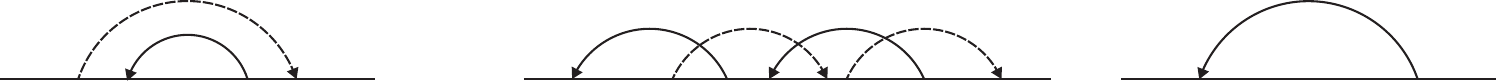}}
        \put(-110,3){3}
        \put(7,3){+ $\frac{1}{3}$}
        \put(162,3){-}
    \end{picture}
    \caption{An example vector in $\GD_4$ as a linear combination of Gauss diagrams.}
    \label{fig:k1_basis}
\end{figure}

\subsection{The GPV Theorem}

Goussarov-Polyak-Viro proved that finite type invariants can be decomposed using a subdiagram map denoted by $\varphi_k$. To understand the map $\varphi_k$ and the GPV theorem, we first need to understand subdiagrams of Gauss diagrams.

\begin{defn}[subdiagram]\label{def:subdiag} Given a Gauss diagram $D$ with $n$ arrows, a \emph{$k$-arrow subdiagram} of $D$ is found by following the steps below.
\begin{enumerate}
    \item Choose $k$ arrows from $D$ to keep.
    \item Remove all other $n-k$ arrows from $D$.
    \item Without changing the relative positions of the remaining $k$ arrows and endpoints, reparameterize the interval so the remaining diagram is a Gauss diagram with $k$ arrows and $2k$ endpoints. This diagram is a $k$-arrow subdiagram of $D$.
\end{enumerate}
\end{defn}

An example subdiagram is shown in Figure \ref{fig:decomposition}. It is important to note that because of the reparameterization, the subdiagram ``forgets" how it sat inside the original diagram. So without additional context, the process to create a subdiagram is not uniquely reversible to recover the original diagram.

\begin{figure}[h]
\centering
    \includegraphics[width=0.85\linewidth]{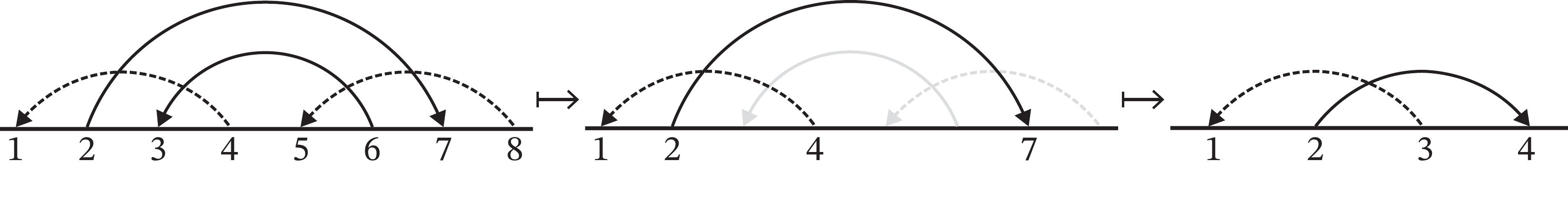}
    \caption{An example Gauss diagram and process to find a single 2-arrow subdiagram.}
    \label{fig:subiagram}
\end{figure}

\begin{defn}[$\varphi_k$]
    Let $\K$ be the set of all long knot diagrams. There is a subdiagram map $\varphi_k:\K\rightarrow \GD_k$ that takes in input a long knot diagram $K$, and outputs the sum of all of the subdiagrams of the Gauss diagrams for $K$ which have $k$ or less arrows.
\end{defn} 
The $\varphi_k$ map should be viewed as a decomposition map, or perhaps a factoring map, that breaks a long knot's Gauss diagram into a sum of its small pieces. See Figure \ref{fig:decomposition} for an example computation of $\varphi_k$ with $k=2$ .

\begin{figure}[h!]
    \centering
    \begin{picture}(190, 70)
        \put(-100,0){\includegraphics[width=0.90\linewidth]{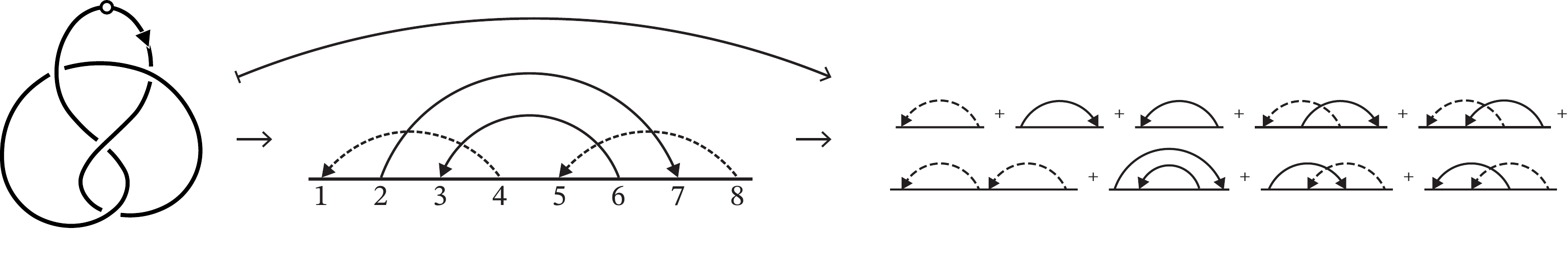}}
        \put(128,32){\scalebox{0.8}{2}}
        \put(33, 66){$\varphi_2$}
    \end{picture}
    \caption{An example computation of $\varphi_2$ applied to the figure-eight long knot.}
    \label{fig:decomposition}
\end{figure}

The map $\varphi_k$ is far from surjective onto $\GD_k$ as, for one, the coefficients of the outputs of $\varphi_k$ are positive integers. For a more interesting example (with signs omitted), let $k\geq 2$ and consider the diagram $\IntersectingArcs$. Suppose there exists some knot $K$ so that $\IntersectingArcs$ is a 2-arrow subdiagram of the Gauss diagram for $K$. Then the Gauss diagram for $K$ must also have $\OneArrow$ as a 1-arrow subdiagram (occurring at least twice), so
$\varphi_k(K)$ would have to contain $\IntersectingArcs$ and $2\OneArrow$ as summands. Therefore $\IntersectingArcs$ is not in the image of $\varphi_k$ as a standalone diagram.

In general, the map $\varphi_k$ is not injective either. The intuitive reason being that if the number of crossings in the knot is some amount larger than $k$, the map $\varphi_k$ does not maintain enough information as to how the crossings are arranged in $K$. For example, let $k=1$,  the trefoil (shown in Figure \ref{fig:UsedKnots}) and an unknot with three positive twists will both output $\OneOtherArrow+2\OneArrow$ under $\varphi_1$.

Goussarov-Polyak-Viro in \cite{GPV}  gave a combinatorial description of all finite type knot invariants via a two-step factorization using the $\varphi_k$ map (called here the GPV Theorem). We choose to follow Roukema's restatement of the GPV Theorem from \cite{Roukema}.

\begin{GPVthm} [Goussarov-Polyak-Viro  \cite{GPV}]\label{thm:FactorsThruPhi} 
A $\Q$-valued knot invariant $f$ is of type $k$ if and only if there is a linear transformation $\omega_f$ on $\mathcal{GD}_k$ such that $f=\omega_f\circ\varphi_{k}$.
\end{GPVthm}
\begin{center}
\begin{tikzcd}
\mathbb{K} \arrow[r,"\varphi_k",swap]\arrow[rr, "f", near start, controls={+(1.7,1.3) and +(-.5,0.8)}]\arrow[rr,"\circlearrowleft",swap, phantom, controls={+(1,1) and +(-.5,0.5)}] &\GD_k\arrow[r,"\omega_f",swap] & \Q\\
\end{tikzcd}
\end{center}

In summary, the GPV Theorem outlines a two-step process used to compute all finite type knot invariants; finite type invariants factor through the same $\varphi_k$ map, followed by a linear map $\omega_f$ that differs depending on the finite type invariant $f$. We will call $\omega_f$ the \emph{GPV map} for invariant $f$.

The GPV map $\omega_f:\GD_k\rightarrow \Q$ is a linear transformation of vector spaces, that is
\[ \omega_f(aD_1+bD_2)=a\omega_f(D_1)+b\omega_f(D_2)\]
for all $a,b\in \Q$ and $D_1,D_2\in \GD_k$.
The focus of this paper is to write formulas to compute $\omega_f$ for different finite type invariants. As a linear transformation,
 $\omega_f$ is completely defined by its outputs on a basis for $\GD_k$.
Our method in Section \ref{sec:conway}  is to determine the values of $\omega_f$ on a basis for $\GD_k$ for low degree coefficients of the Conway polynomial.

Before we discuss the Conway polynomial, in Section \ref{sec:linking} we will first look at the linking number, which is not a knot invariant, but a link invariant.
For link invariants, the space of Gauss diagrams is a bit bigger and more complicated. So in the next section, we investigate Gauss diagrams for links and consider the GPV theorem for the linking number.

\section{The GPV map for the linking number}\label{sec:linking}

The linking number, denoted $lk$, is a link invariant that counts how many times two components link with each other. 
To compute the linking number of two oriented components $K_1$ and $K_2$ of a link, ignore all crossings except for the crossings between the two chosen components. 
The linking number, $lk(K_1,K_2)$, is the number of positive crossings between the two components minus the number of negative crossings between the two components, all times one half.

It can be quickly checked that the linking number is a finite type invariant of type 1 (vanishes on links with 2 or more double points).
The GPV theorem is stated for \emph{knot} invariants, but it turns out the theorem is true for the linking number after some slight modification. 
To understand the GPV theorem for links, we need to understand  Gauss diagrams for a link and the vector space of link Gauss diagrams.

As with knots, we will work with oriented long link diagrams where the starting point of the parameterization for each component is marked with a point on an arc of each component diagram.
A Gauss diagram for a long link has one parameterized interval for each component, arrows between the two intervals represent a crossing between two components, and the usual arrows on the same interval represent crossings within one component. As for Gauss diagrams for knots, there are two types of arrows, dashed and solid, for positive and negative crossings. See Figure \ref{fig:link_gd} for an example Gauss diagram for a link with three components.

\begin{figure}[h]
    \centering
    \includegraphics[width=.4\linewidth]{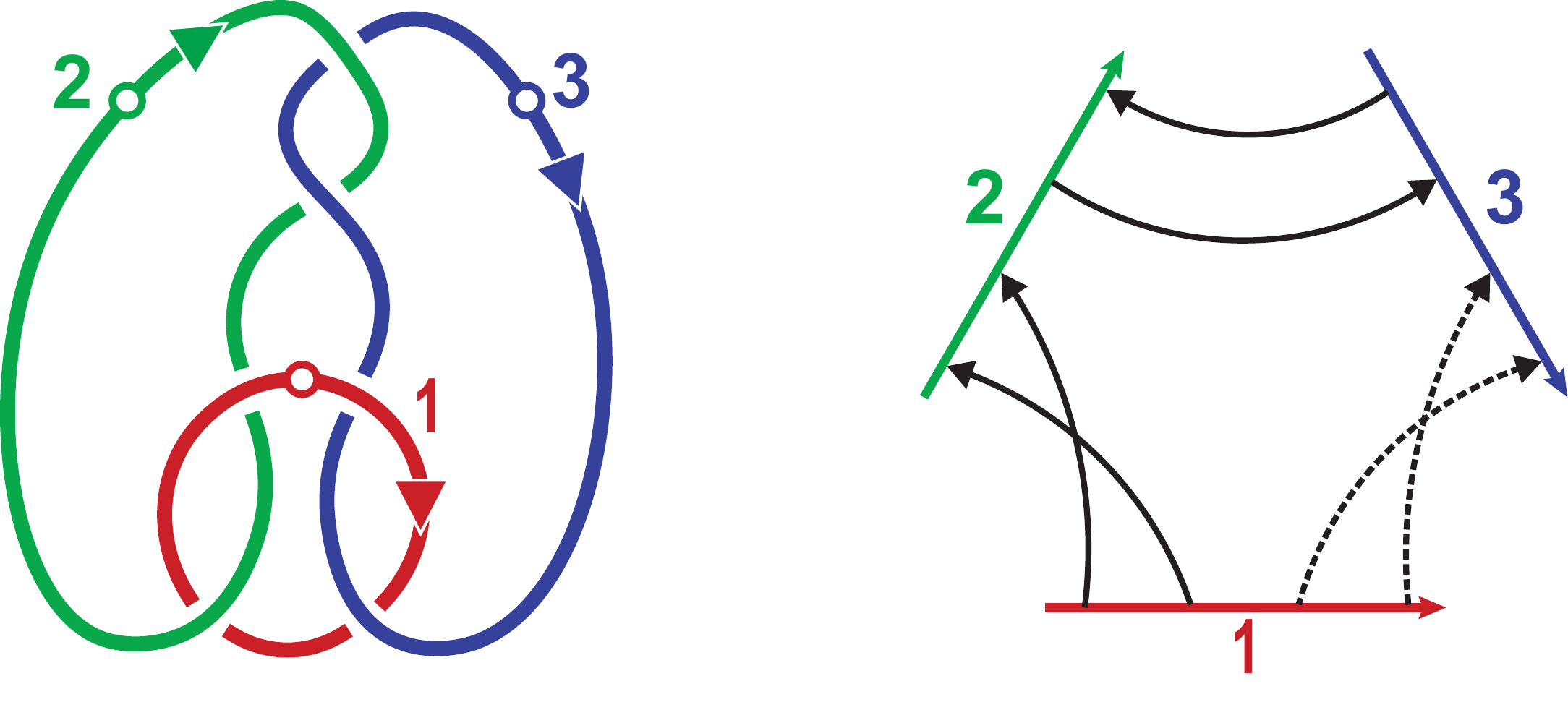}
    \caption{A long link with three components and its Gauss diagram.}
    \label{fig:link_gd}
\end{figure}

Since links can have an arbitrary number of components, the Gauss diagrams for links can have an arbitrary number of parameterized intervals. We will assume the components of a link come with an ordering, or a numbering, and we number the parameterized intervals of the Gauss diagram correspondingly. For ease in computation, when computing the linking number, we can only compute the linking number of two components at a time. So, we will assume the two components for which we are computing the linking number are numbered first in the ordering. That is, given a link $L$ with components $L_1,\cdots , L_k$, we will assume (possibly up to a renumbering) that we are interested in computing the linking number between $L_1$ and $L_2$.

Let $\LGD_k$ denote the vector space of $\Q$-linear combinations of link Gauss diagrams with $k$-arrows (there is no bound on the number of link components). The Gauss diagrams in $\LGD_k$ are much more complicated than those in $\GD_k$, as they may have more than one parameterized interval (one for each link component), and each parameterized interval is numbered.
Notice that $\GD_k$ is a subspace of $\LGD_k$, as every knot can be viewed as a link with only one component that is numbered by a 1.
See Figure \ref{fig:ex_LGD_2} for an example element in $\LGD_3$.
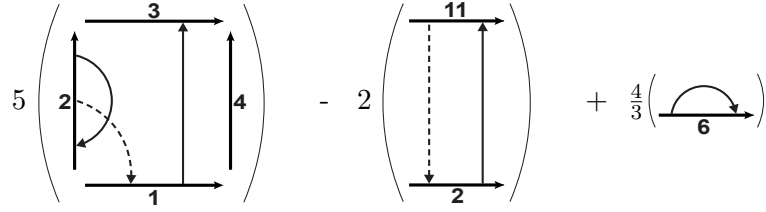
\begin{figure}[]
    \centering
    \begin{picture}(190, 70)
        \put(-40,0){\includegraphics[width=0.60\linewidth]{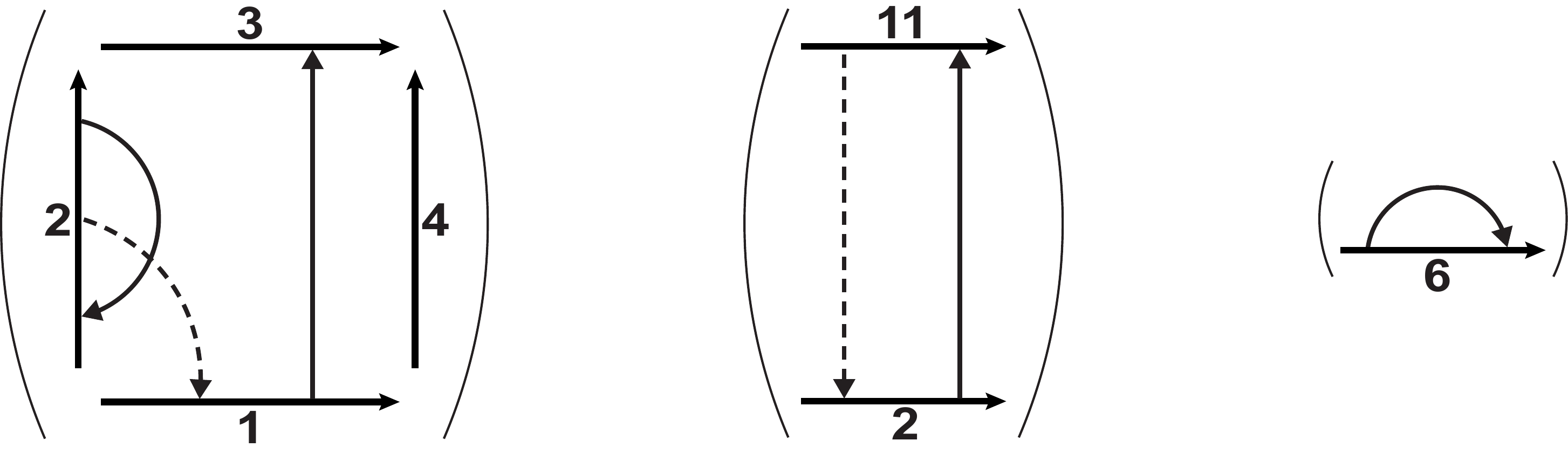}}
        \put(-50,39){5}
        \put(66,39){- \; 2}
        \put(166,39){+\space\space$\frac{4}{3}$}
    \end{picture}
    \caption{An example element in $\LGD_3$}
    \label{fig:ex_LGD_2}
\end{figure}

Let $\mathcal{L}$ denote the space of all long link diagrams. Subdiagrams of link Gauss diagrams are found by following the same procedure as in Definition \ref{def:subdiag}, where each component interval is reparameterized or deleted entirely if there are no arrows on that component. Thus the subdiagram map $\varphi_k$ extends to $\mathcal{L}$ in a natural way, and we can view $\varphi_k:\mathcal{L}\rightarrow \LGD_k$.

Since the linking number is of type 1, we focus our attention on $\LGD_1$. As a vector space, $\LGD_1$ has a basis consisting of all diagrams of the form
\begin{center}
 \linkArcsolidup{n}{m}, \linkArcdashup{n}{m}, \RightArrow{n}, \RightArrowdashed{n} \LeftArrow{n}, and \LeftArrowdashed{n} 
 \end{center}
 for all natural numbers $n\neq m$, which is a countably infinite basis. The numeric labels in front of the parameterizing line are not coefficients but recording the label of the component.

Since the linking number is computed by locally summing values (signs of crossing) one crossing at a time, it is straightforward to identify the GPV map $\omega_{lk}$.
Since crossings between the same component do not contribute to the linking number, $\omega_{lk}$ must map all one-arrow diagrams on the same component to 0; $\omega_{lk}(\OneArc)=0$ for any choice of signed arrow, orientation of the arrow head, and numbered labeling of the component.
Any positive crossing between the first two labeled components contributes $\frac{1}{2}$ to the linking number, and any negative crossing between the first two labeled components contributes $-\frac{1}{2}$ to the linking number. Thus we define $\omega_{lk}:\LGD_1\rightarrow \Q$ on a basis for $\LGD_1$ as follows:

\begin{align*}   
    \omega_{lk}\left(\linkArcsolidup{1}{2}\right)&=\omega_{lk}\left (\linkArcsoliddown{1}{2}\right)=-\frac{1}{2}\\
    \omega_{lk}\left(\linkArcdashup{1}{2}\right)&=\omega_{lk}\left(\linkArcdashdown{1}{2}\right)=\frac{1}{2}\\
    \omega_{lk}\left(\linkArcsolidup{n}{m}\right)&=\omega_{lk}\left (\linkArcdashup{n}{m}\right)=0, \\
    \omega_{lk}\left(\LeftArrow{n}\right)&=\omega_{lk}\left (\RightArrow{n}\right)=\omega_{lk}\left(\LeftArrowdashed{n}\right)=\omega_{lk}\left (\RightArrowdashed{n}\right)=0\\
    \end{align*}
    
where $n$ and $m$ are any natural numbers so that $\{n,m\}\neq \{1,2\}$.
By construction, we get the following result.
\begin{thm}\label{thm:main_linking}
    The GPV theorem is true for the linking number using the map $\omega_{lk}$. That is, Let $L$ be a link with labeled components $L_1, \cdots, L_k$, then  $lk(L_1,L_2)=\omega_{lk}\circ \varphi_1(L)$.

\end{thm}

\section{The GPV maps  for low-degree coefficients of the Conway polynomial}\label{sec:conway}

The Conway polynomial, also called the Alexander-Conway polynomial, of a knot $K$, denoted $\nabla(K)$, is a polynomial in variable $z$ which can be computed using the Conway skein relation as follows. Let $K$ be an oriented knot or link diagram, and let $x$ be a crossing in $K$. There are three related diagrams to $K$ found by replacing the crossing $x$ with a positive crossing \pcross, a negative crossing \ncross, or smoothing the crossing \smooth \text{ }, which reduces the complexity of the diagram. Denote these three diagrams by $K_+$, $K_-$, and $K_0$ respectively (the original $K$ will be equal to either $K_+$ or $K_-$).
The Conway skein relations are given by the following two equations
\begin{align*}
\nabla(\text{unknot})=\nabla(\text{unlink})&=1\\
\nabla(K_+)-\nabla(K_-)&=z\nabla(K_0)
\end{align*}

In practice, to compute $\nabla(K)$, one starts with a diagram of $K$ and systematically smooths one crossing at a time (creating a resolving tree) until all that is left is a collection of unlinked unknots. Using the skein relation, one recursively builds the polynomial for $K$ by working up the tree.

For a knot $K$, let $V_k(K)$ be the $k$'th coefficient of $z$ of the Conway polynomial $\nabla(K)$. 
Bar-Natan proved that $V_k$ is a finite type invariant of type $k$, see \cite{BN95}[example on page 19] see also \cite{BN1}.
It follows from the GPV theorem that for each $k$, there exists a map $\omega_{V_k}:\GD_k\rightarrow \Q$ so that $V_k=\omega_{V_k}\circ \phi_k$.
In this section, we compute the $\omega_{V_k}$ for several values of $k$.

Firstly, Kauffman proved that the Conway polynomial of a knot has only even powers of $z$ \cite{KAUF81}[Proposition 4.1]. 

\begin{prop}
    For $k$ odd, $\omega_{V_k}$ is the zero map.
\end{prop}

\begin{proof}
    Since $\nabla$ has coefficients of $0$ for every odd-degree term, $V_k=0$ is the zero invariant. So,  $\omega_{V_k}$ must also be the zero map.
\end{proof}

When $k=0$, $\GD_0$ is generated by the empty diagram and $\omega_{V_0}$ is also very easy to compute.

\begin{prop}
The map $\omega_{V_0}$ is the coefficient map that sends $t*(\text{empty diagram})\mapsto t$.
      
\end{prop}\begin{proof}
For $U$ the unknot, $\nabla(U)=1$, so $V_0(U)=1$. Since $\varphi_0(U)$ is the empty diagram, then \[1=\nabla(U)=\omega_{V_0}\circ \varphi_0(U)=\omega_{V_0}(\text{empty diagram}).\]

Thus, $\omega_{V_0}$ is the map $\GD_0\rightarrow \Q$ that sends $t*(\text{empty diagram})\mapsto t$.

\end{proof}

\subsection{Quadratic Conway coefficient}
The quadratic coefficient of the Conway polynomial for a knot $K$ is called $V_2(K)$ and is a finite type invariant of type 2.

When $k=2$, $\GD_k$ is 52-dimensional with the basis shown in Figure \ref{fig:Two arrow diagrams}. This basis consists of 48 double-arrow diagrams and 4 one-arrow diagrams from $\GD_1$.  The 48 double-arrow diagrams are evenly divided into three different categories: disjoint, interleaved, and nested. 

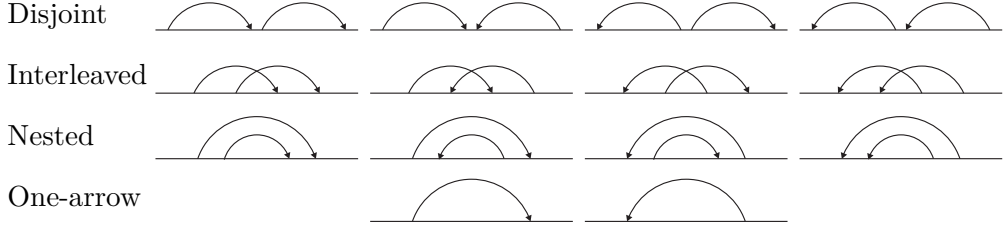
\begin{figure}
    \centering
    \begin{picture}(290, 100)
     \put(0,0){ \includegraphics[width=0.70\linewidth]{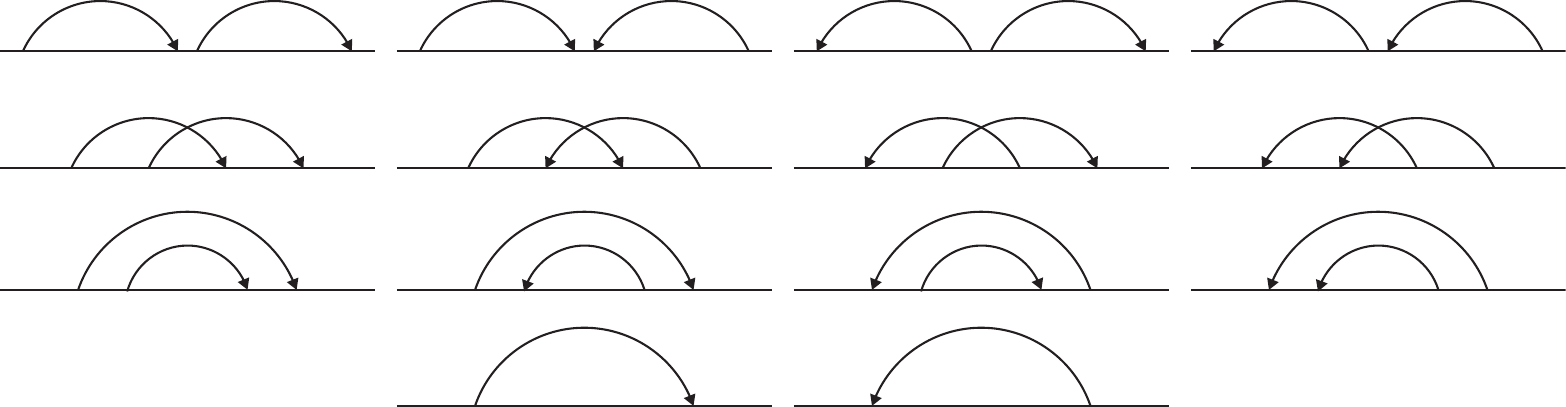}}
     \put(-52,75){Disjoint}
     \put(-52,52){Interleaved}
     \put(-52,29){Nested}
     \put(-52,6){One-arrow}
    \end{picture}
    
    \caption{Types of basis diagrams for $\GD_2$. Each arrow can be positive or negative (dashed or solid), which totals to 52 distinct basis diagrams. }
    \label{fig:Two arrow diagrams}
\end{figure}

The interleaved diagrams play an important role in the computation of $\omega_{V_2}$, so we give a naming convention for these diagrams. We name each interleaved diagram by the direction and sign of its leftmost and rightmost arrows. An interleaved diagram consists of two arrows with endpoints at positions 1, 2, 3, 4 on the parameterizing line, one arrow connecting positions 1 and 3 and the other connecting 2 and 4. We label each diagram $XY^{ab}$, where X and Y record whether the leftmost and rightmost arrows point right (R) or left (L), and a and b record the sign of the leftmost and rightmost arrows, respectively. For Example, $LR^{+-}$ denotes the diagram with a positive arrow pointing from positions 3 to 1 and a negative arrow pointing from positions 2 to 4.

In the following theorem, we compute the values of $\omega_{V_2}$ on a basis for $\GD_2$, \emph{almost}. We are able to show that most basis vectors map to zero, some map to determined values, and several satisfy a 1-parameter relation.
After the theorem and proof, we discuss this parameter and address the question of uniqueness of $\omega_{V_2}$.

\begin{thm}\label{thm:main_Conway}
    The map $\omega_{V_2}:\GD_2 \rightarrow \Q$ satisfies the following properties on a basis for $\GD_2$:
    \begin{itemize}
        \item Every one-arrow diagram maps to 0.
        \item Every nested two-arrow diagram maps to 0.
        \item Every disjoint two-arrow diagram maps to 0.
        \item On the interleaved two-arrow diagrams, $RR^{++},RR^{+-},RR^{-+},RR^{--},LL^{++}$ and $LL^{--}$ are all mapped to 0, and 
        the remaining diagrams satisfy a 1-parameter family of equations in $t$ given by the following parametric equations:
        \begin{align*}
            \omega_{V_2}(RL^{++})&=\omega_{V_2}(RL^{+-})=\frac{t}{2} & \omega_{V_2}(RL^{-+})&=\omega_{V_2}(RL^{--})=-\frac{t}{2} & &\\
            \omega_{V_2}(LR^{++})&=-\omega_{V_2}(LR^{+-})=1-\frac{t}{2}  & \omega_{V_2}(LR^{--})&=-\omega_{V_2}(LR^{-+})=1+\frac{t}{2}&  &\\
            \omega_{V_2}(LL^{-+})&= - \omega_{V_2}(LL^{+-}) =t
        \end{align*}
        
    \end{itemize}
\end{thm}

\begin{proof}

     First, we show that all one-arrow diagrams map to 0 by $\omega_{V_2}$. 
     Let $D_{\uparrow}$ be the Gauss diagram for an unknot with one crossing, i.e., start with the standard unknot and add one kink. So $D_{\uparrow}$ is a diagram with one arrow.
      The Conway Polynomial of the unknot is 1,  so  $\omega_{V_2}(D_\uparrow)=0$.
      For example,  Figure \ref{fig:Zero Map}(top) shows the unknot that has \OneOtherArrow as its Gauss diagram.
      Any one-arrow Gauss diagram arises as the Gauss diagram for an unknot with one kink by choosing different orientations and positive versus negative kinks to add to the unknot.
      So, every one-arrow diagram in the basis for $\GD_2$ maps to 0 under $\omega_{V_2}$.

      \begin{equation}\label{eq:OV_2_single}
          \omega_{V_2}(\text{one-arrow diagram})=0
      \end{equation}
      
    \begin{figure}[h!]
        \centering
        \begin{picture}(190, 70)
            \put(-100, 0){\includegraphics[width=0.85\linewidth]{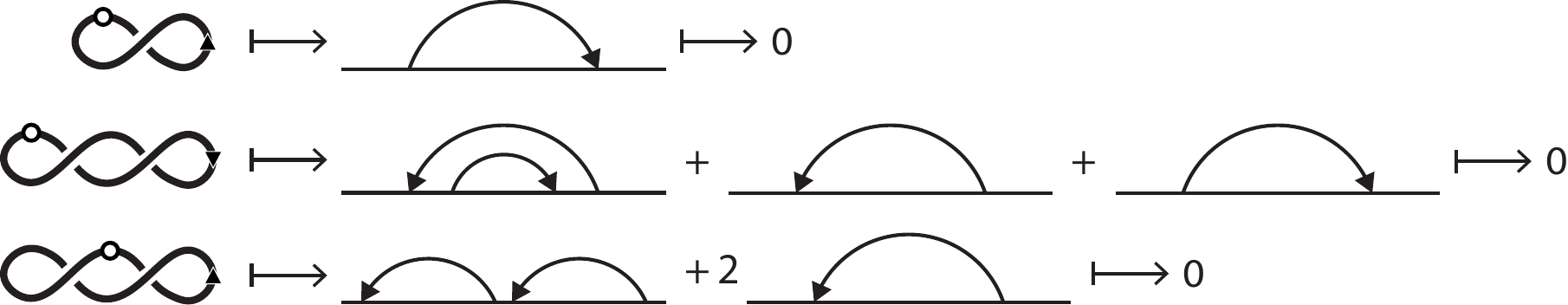}}
            \put(-36, 70){$\varphi_2$}
            \put(69, 71){$\omega_{V_2}$}
            \put(-36, 41){$\varphi_2$}
            \put(260, 42){$\omega_{V_2}$}
            \put(-36, 12){$\varphi_2$}
            \put(170, 14){$\omega_{V_2}$}
        \end{picture}
        \caption{Example of nested and disjoint crossings formed by unknots.}
        \label{fig:Zero Map}
    \end{figure}
    
    Next, we prove all nested and disjoint diagrams also map to 0 by $\omega_{V_2}$. 
    Notice that any knot with two crossings is the unknot, and thus the Gauss diagram for any two-crossing unknot will have exactly two arrows. Let $D_{\uparrow\uparrow}$ be the Gauss diagram for a two-crossing unknot. So, $\varphi_2(D_{\uparrow\uparrow})$  will output exactly one two-arrow subdiagram (which is the entire diagram $D_{\uparrow\uparrow}$) and two one-arrow subdiagrams, which we denote $D_{\uparrow}$ and $D'_{\uparrow}$. See Figure \ref{fig:Zero Map} (middle and bottom) for two examples. Thus, for any two-arrow Gauss diagram, we get the following equation
\begin{equation}\label{eq:phi_2_arrow}
    \varphi_2(D_{\uparrow\uparrow})= D_{\uparrow\uparrow}+D_{\uparrow}+D'_{\uparrow}.
\end{equation}
    We claim that all nested and disjoint diagrams arise as Gauss diagrams for two crossing unknots. Figure \ref{fig:nested_disjoint_unknots} shows four diagrams, each with two intersections. Choosing different orientations and any choice of positive or negative crossings for the intersections yields 32 different two-crossing unknots. A quick verification shows that all nested and disjoint diagrams occur as the Gauss diagrams for these unknots.

    \begin{figure}[h!]
    \centering
    \includegraphics[width=.55\linewidth]{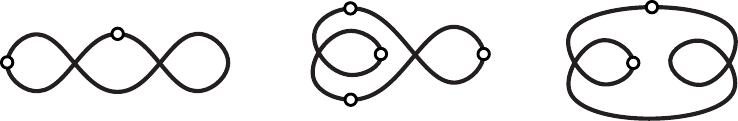}
    \caption{Types of unknots that form all 2-arrow nested and disjoint Gauss diagrams}
    \label{fig:nested_disjoint_unknots}
\end{figure}

    Let $D_{\uparrow \uparrow}$ be either a nested or disjoint two-arrow Gauss diagram. As described above, $D_{\uparrow \uparrow}$ is the Gauss diagram for a two-crossing unknot, and 
    \begin{align*}
0=V_2(\text{unknot})&\stackrel{GPV}{=}\omega_{V_2}(\varphi_2(D_{\uparrow\uparrow}))\\
&\stackrel{(\ref{eq:phi_2_arrow})}{=}\omega_{V_2}(D_{\uparrow \uparrow}+D_{\uparrow } +D'_{\uparrow} )\\
&=\omega_{V_2}(D_{\uparrow \uparrow})+\omega_{V_2}(D_{\uparrow } )+\omega_{V_2}(D'_{\uparrow} )\\
&  \stackrel{\ref{eq:OV_2_single}}{=}\omega_{V_2}(D_{\uparrow \uparrow})+0+0
    \end{align*}
Thus, for any nested or disjoint two-arrow Gauss diagram $D_{\uparrow \uparrow}$, we have that
    \[
        \omega_{V_2}(D_{\uparrow \uparrow})=0.
    \]

    Now we have shown that all one-arrow, nested two-arrow, and disjoint two-arrow diagrams are mapped to zero by $\omega_{V_2}$. What remains are 16 interleaved diagrams which require nontrivial computation to determine their output values under $\omega_{V_2}$.

    Using the previously defined naming conventions, we choose the following order of the interleaved basis diagrams:
     \vspace{2.5mm}
    
        {\scriptsize
        $RR^{++}, RR^{+-}, RR^{-+}, RR^{--}, RL^{++}, RL^{+-}, RL^{-+}, RL^{--}, LR^{++}, LR^{+-}, LR^{-+}, LR^{--}, LL^{++}, LL^{+-}, LL^{-+}, LL^{--}$}
    
\vspace{2.5mm}
We build a system of linear equations for the outputs of $\omega_{V_2}$ by computing the Conway polynomial for 16 knots and expressing the outputs in terms of this ordered basis. 
We use the long knots shown in Figure \ref{fig:UsedKnots}: the trefoil, figure eight, $5_2$, and 13 unknots.
Because we are working with long knots, changing the basepoint results in different Gauss diagrams, but the Conway polynomial remains the same. We use this fact on the unknot to form several distinct Gauss diagrams from a single knot diagram, giving us additional linear equations from the same knot. Each basepoint used is marked by a circle in Figure \ref{fig:UsedKnots}. For each knot $K$ in Figure \ref{fig:UsedKnots}, we compute $\varphi_2(K)$ and expand the result in terms of the basis for $GD_2$. From the GPV theorem, we know that $\omega_{V_2}(\varphi_2(K)) = V_2(K)$, which gives a desired linear equation.
    
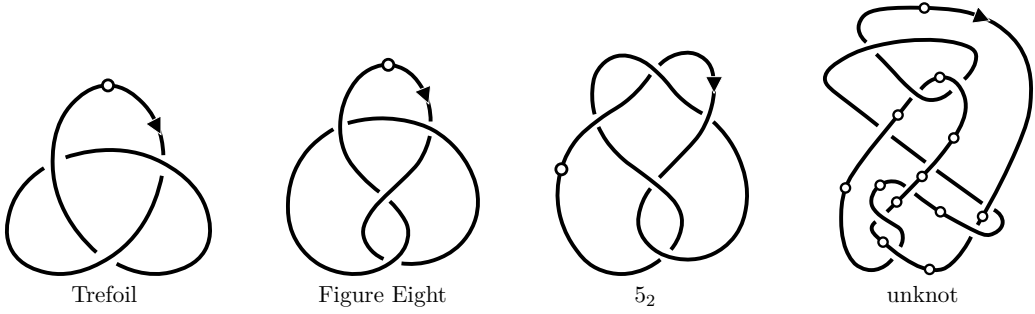
\begin{figure}[]
    \centering
    \begin{picture}(190, 110)
        \put(-100,10){\includegraphics[width=0.85\linewidth]{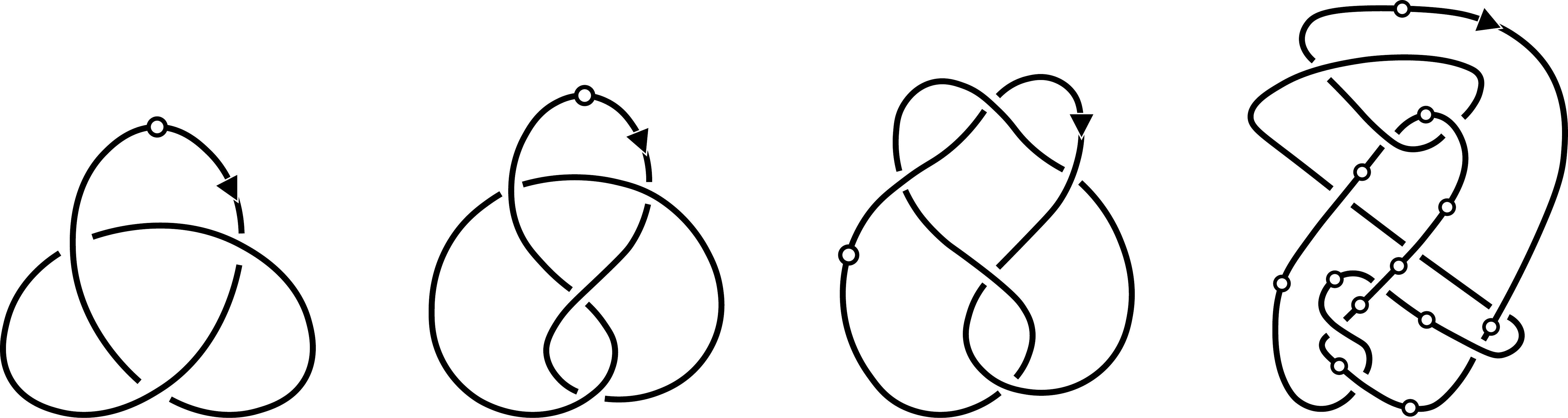}}
        \put(-75,0){\scalebox{0.8}{Trefoil}}
        \put(18,0){\scalebox{0.8}{Figure Eight}}
        \put(137,0){\scalebox{0.8}{$5_2$}}
        \put(232, 0){\scalebox{0.8}{unknot}}
    \end{picture}
    \caption{16 long knots used to build the linear system.  The unknot is ``the culprit" from \cite{HK}.}
    \label{fig:UsedKnots}
\end{figure}

For example, let $T$ be the trefoil in Figure \ref{fig:UsedKnots}. The Conway polynomial for $T$ has a quadratic coefficient of 1, so $V_2(T)=1$. On the other hand, $$\omega_{V_2}(\varphi_2(T))=\omega_{V_2}(RR^{--})+\omega_{V_2}(RL^{--})+\omega_{V_2}(LR^{--}),$$
which implies that $$\omega_{V_2}(RR^{--})+\omega_{V_2}(RL^{--})+\omega_{V_2}(LR^{--})=1.$$

For every long knot diagram in Figure \ref{fig:UsedKnots}, we get such a linear equation. We organize this system into an augmented matrix where each column corresponds to an interleaved basis diagram in the ordered basis and the final column is that knot's quadratic Conway coefficient.
So for example, the trefoil contributes the row $(0,0,0,1,0,0,0,1,0,0,0,1,0,0,0,0|1)$.
The first 13 rows are from the unknot, the 14th row is the trefoil, followed by the figure-eight and lastly the $5_2$ knot is the 16th row.
Before row reduction, the system is the following matrix.

\newpage

\begin{figure}[h]
    \centering
    \includegraphics[width=0.9\linewidth]{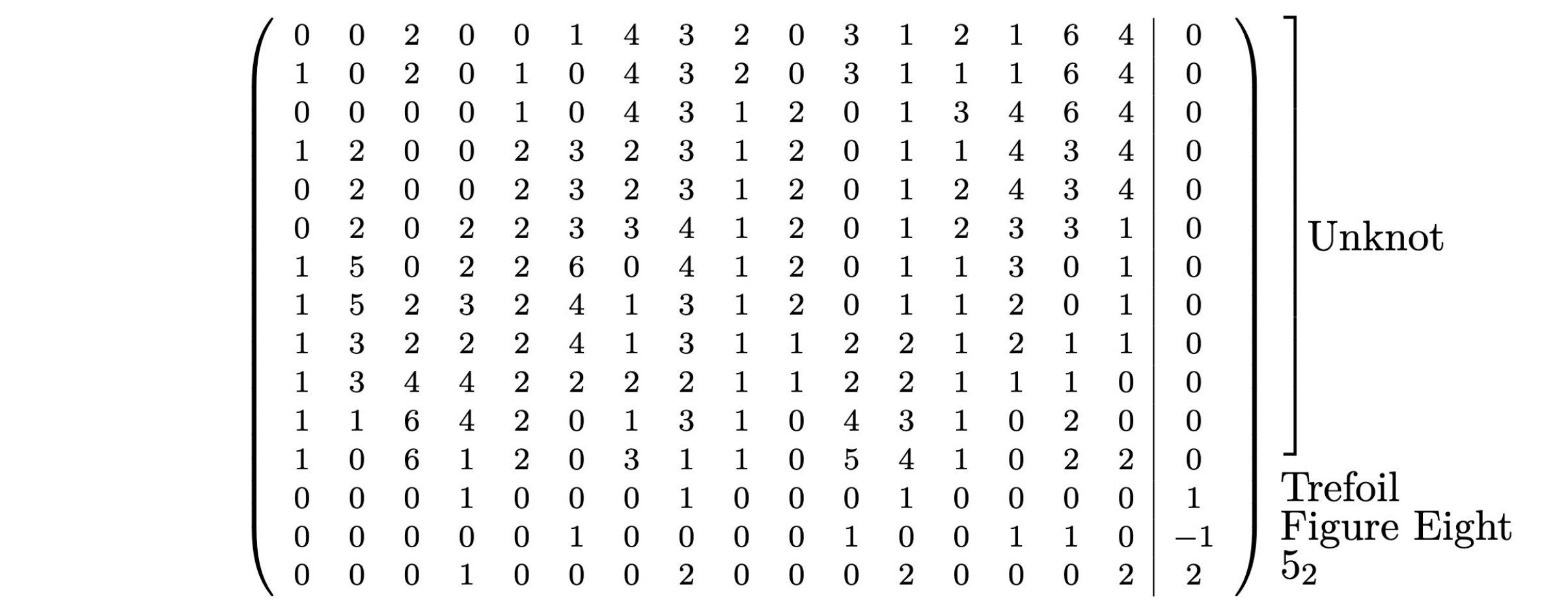}
    \label{fig:matrix}
\end{figure}

Row reducing the matrix leaves one free variable, so $\omega_{V_2}$ is determined on the interleaved diagrams up to one scalar. After row reduction, we found the following parametric equation of our solution:

\begin{equation*}
    \omega_{V_2}(RR^{++}), \omega_{V_2}(RR^{+-}), \omega_{V_2}(RR^{-+}), \omega_{V_2}(RR^{--}), \omega_{V_2}(LL^{++}), \omega_{V_2}(LL^{--})=0
\end{equation*}
\begin{align*}
    \omega_{V_2}(RL^{++})&=\omega_{V_2}(RL^{+-})=\frac{t}{2} &\omega_{V_2}(RL^{-+})&=\omega_{V_2}(RL^{--})=-\frac{t}{2} & &\\
    \omega_{V_2}(LR^{++})&=-\omega_{V_2}(LR^{+-})=1-\frac{t}{2}  &\omega_{V_2}(LR^{--})&=-\omega_{V_2}(LR^{-+})=1+\frac{t}{2}&  &\\
    \omega_{V_2}(LL^{-+})&= - \omega_{V_2}(LL^{+-}) =t
\end{align*}

\end{proof}

Now, let's address the question: is $\omega_{V_2}$ unique? As stated in the theorem, there is a parameter, called $t$, for which different values of $t$ could potentially give different $\omega_{V_2}$ maps that still satisfy the GPV theorem.
The results of our computation are inconclusive as to whether there is truly a one-parameter family of GPV maps for the quadratic coefficient of the Conway polynomial, or if we were not able to completely determine the system.

However, we don't expect $\omega_{V_2}$ to be unique! The GPV theorem does not guarantee uniqueness in any sense, and we conjecture that Theorem \ref{thm:main_Conway} accurately describes all possible $\omega_{V_2}$ maps.
For the reminder of this section, we explain some methods we attempted to use to determine a unique value of $t$. We encourage the reader to continue this pursuit to either determine a unique value of $t$, or verify that there are indeed infinitely many possible choices for $t$.


One computational method to prove there is only one GPV map -- so exactly one choice of $t$ that is determined by $V_2$ giving a unique $\omega_{V_2}$ -- would be to find one more knot whose expansion in terms of the 16 interleaved basis adds a linearly independent row to the augmented matrix shown in the proof above. 
We checked an additional 75 long knots including one long knot for each knot with 7 or less crossings, as well as $8_1$, $8_2$, $8_3$, $8_4$, $8_5$, $8_{11}$, $8_{13}$, $8_{16}$, $10_{18}$, $10_{41}$, $10_{46}$, $10_{134}$, $10_{139}$, the whitehead double of the figure eight knot, 30 long unknots, and different choices of starting points for several of these listed. We found all of these knots to be dependent on the knots we have presented in Figure \ref{fig:UsedKnots}.

Another method would be to exploit properties of the Conway polynomial directly.
The Conway Polynomial is proven to be invariant under mirror images of knots \cite{KAUF81}. There is a ``mirror image" map $M : \mathbb{K} \rightarrow  \mathbb{K}$ sending a long knot to its mirror image (change all positive crossings to negative and vice versa).
The fact that the Conway polynomial is invariant under mirror imaging can be expressed in the equation $\nabla\circ M=\nabla$, which is also true for $V_k$, the $k$'th coefficient of $\nabla$, $V_k\circ M=V_k$. This relationship can be expressed by the commutative diagram below.

\begin{equation}\label{small_diagram}
    \begin{tikzcd}
\mathbb{K} \arrow[r,"V_k"]\arrow[d, "M", swap]& \Q\\
\mathbb{K}\arrow[ru,"V_k", swap]\arrow[ru,"\circlearrowleft"]&
\end{tikzcd}
\end{equation}

The mirror image map on $\mathbb{K}$ induces a mirror imaging map, also denoted $M$, on $\GD_k$ for any $k$, which swaps the sign of all arrows in the Gauss diagram.
Putting all of this information together with the GPV Theorem, we get the following commuting diagram where the outer triangle is the diagram in Equation \ref{small_diagram}.

    \begin{equation}\label{big_diagram}
    \begin{tikzcd}
\mathbb{K} \arrow[r,"\varphi_k"]\arrow[d, "M", swap]\arrow[rr, "V_k",  controls={+(1.7,1.3) and +(-.5,0.8)}]\arrow[rr,phantom,near start, "\circlearrowleft",swap,  controls={+(1.7,1) and +(-.5,0.2)}]&\GD_k\arrow[r,"\omega_{V_k}"]\arrow[d,"M",swap] & \Q\\
\mathbb{K}\arrow[r,"\varphi_k"]\arrow[rru, "V_k", swap, controls={+(+1.7,-1.3) and +(0,-1.5)}]\arrow[rru,phantom, near start,  "\circlearrowleft", controls={+(+1.7,-1) and +(-.5,-1.2)}]\arrow[ru, phantom, "\circlearrowleft"]&\GD_k\arrow[ru,"\omega_{V_k}",swap] &
\end{tikzcd}
\end{equation}

It is tempting to think that the commutativity of three of the faces implies the commutativity of the last face, which is the triangle below.

\begin{equation}\label{small_diagram_2}
    \begin{tikzcd}
\GD_k \arrow[r,"\omega_{V_k}"]\arrow[d, "M", swap]& \Q\\
\GD_k\arrow[ru,"\omega_{V_k}", swap]\arrow[ru,"?"]&
\end{tikzcd}
\end{equation}

If this triangle commutes, we could conclude that $\omega_{V_k}\circ M=\omega_{V_k}$ on $\GD_k$.
However, from the diagram in Equation \ref{big_diagram}, all we can conclude is the commutativity of $\omega_{V_k}$ and $M$ on the image of $\varphi_k$.
That is, $\omega_{V_k}\circ M(D)=\omega_{V_k}(D)$ when $D$ is the image of a knot under $\varphi_k$.
\begin{equation}\label{small_diagram_3}
    \begin{tikzcd}
im(\varphi_k) \arrow[r,"\omega_{V_k}"]\arrow[d, "M", swap]& \Q\\
im(\varphi_k)\arrow[ru,"\omega_{V_k}", swap]\arrow[ru,"\circlearrowleft"]&
\end{tikzcd}
\end{equation}
Recall that $\K$ is the set of all long knot diagrams. To be explicitly clear, these are knot diagrams with a marked starting point and a chosen orientation for the direction of the parameterization.
The vector space $\GD_k$ is the set of all Gauss diagrams with $k$ or less arrows. 
The subdiagram map $\varphi_k:\K\rightarrow \GD_k$ is not surjective, as was described in Section \ref{sec:GDandGPV}. So, $im(\varphi_k)$ is a proper subset of $\GD_k$.

Why does this distinction matter? Well, let's focus on the case of $k=2$ for the quadratic coefficient.
Suppose we knew that the diagram in Equation \ref{small_diagram_2} commutes and that $\omega_{V_2}\circ M=\omega_{V_2}$ on all of $\GD_2$.
Then we can conclude that $\omega_{V_2}$ also preserves mirror image on every interleaved basis diagram (which these diagrams themselves are not in the image of $\varphi_2)$. 
The effect $M$ has on the interleaved basis diagrams is swapping the sign of every arrow, but not changing the arrow orientation. For example, we can see that $M(RL^{+-})=RL^{-+}$, which implies that $\omega_{V_2}(RL^{+-})=\omega_{V_2}(RL^{-+})$. This is an equation that can be added to the matrix the proof of Theorem \ref{thm:main_Conway} as a row of zeroes with 1 in the $RL^{+-}$ position and -1 in the $RL^{-+}$ position. On the interleaved basis of $\GD_2$, $M$ satisfies the following relations, each of which induces an equality that $\omega_{V_2}$ satisfies.

\begin{multicols}{3}
$M(RR^{++})=RR^{--}$

$M(RR^{+-})=RR^{-+}$

$M(RL^{++})=RL^{--}$

$M(RL^{+-})=RL^{-+}$

$M(LR^{++})=LR^{--}$

$M(LR^{+-})=LR^{-+}$

$M(LL^{++})=LL^{--}$

$M(LL^{+-})=LL^{-+}$
\end{multicols}

Adding these equations to our linear system completely determines the system and gives rise to the value of $t=0$. So, if one could show that $\omega_{V_2}\circ M=\omega_{V_2}$ on all of $\GD_k$, then there is a unique $\omega_{V_2}$.

However, one must proceed with caution.
We can also go through the same process with orientation reversal. The Conway Polynomial is proven to be invariant under orientation reversal \cite{KAUF81}. For example, if $\overrightarrow{K}$ and $\overleftarrow{K}$ are the same long knot diagram with opposite choices of orientation, then $\nabla(\overrightarrow{K})=\nabla(\overleftarrow{K})$. There is an orientation reversal map  denoted $OR$ and we get the analogous commutative diagram.

    \begin{equation}\label{big_diagram}
    \begin{tikzcd}
\mathbb{K} \arrow[r,"\varphi_k"]\arrow[d, "OR", swap]\arrow[rr, "V_k",  controls={+(1.7,1.3) and +(-.5,0.8)}]\arrow[rr,phantom,near start, "\circlearrowleft",swap,  controls={+(1.7,1) and +(-.5,0.2)}]&\GD_k\arrow[r,"\omega_{V_k}"]\arrow[d,"OR",swap] & \Q\\
\mathbb{K}\arrow[r,"\varphi_k"]\arrow[rru, "V_k", swap, controls={+(+1.7,-1.3) and +(0,-1.5)}]\arrow[rru,phantom, near start,  "\circlearrowleft", controls={+(+1.7,-1) and +(-.5,-1.2)}]\arrow[ru, phantom, "\circlearrowleft"]&\GD_k\arrow[ru,"\omega_{V_k}",swap] &
\end{tikzcd}
\end{equation}

We arrive at the same blockade, and we can conclude that $\omega_{V_k}\circ OR=\omega_{V_k}$ on $im(\varphi_k)$ and not necessarily on all of $\GD_k$.
Focusing on $k=2$, if we assume that the orientation reversing map  commutes with $\omega_{V_2}$ on all of $\GD_2$, this leads to an inconsistent system, contradicting the GPV theorem. So orientation reversing does not commute with $\omega_{V_k}$ on all of $\GD_2$.

\bibliography{GPV.bib}{}
\bibliographystyle{plain}

\end{document}